\documentclass[11pt]{article}

\usepackage{amsmath,amsthm,amsfonts,amssymb,amscd}
\usepackage{multirow,booktabs}
\usepackage[table]{xcolor}
\usepackage{fullpage}
\usepackage{lastpage}
\usepackage{enumerate}
\usepackage{fancyhdr}
\usepackage{mathrsfs}
\usepackage{wrapfig}
\usepackage{setspace}
\usepackage{calc}
\usepackage{multicol}
\usepackage{cancel}
\usepackage[margin=3cm]{geometry}
\usepackage{amsmath}
\usepackage{hyperref}

\usepackage{empheq}
\usepackage{framed}
\usepackage[most]{tcolorbox}
\usepackage{xcolor}
\colorlet{shadecolor}{orange!15}
\theoremstyle{definition}

\newtheorem{thm}{Theorem}[section]
\newtheorem{defn}[thm]{Definition}
\newtheorem{prop}[thm]{Proposition}

\newtheorem{lem}[thm]{Lemma}

\newtheorem{rem}[thm]{Remark}
\newtheorem{exam}[thm]{Example}

\newtheorem{asu}{Assumption}
\newcounter{subassumption}[asu]

\makeatletter
\renewcommand{\p@subassumption}{\theasu}
\makeatother

\def\pp#1{ \left(#1\right) }

\def\pc#1{ \left\{#1\right\} }

\newcommand{\cl}{\mathcal}
\newcommand{\bb}{\mathbb}

\newcommand{\wt}{\widetilde}
\newcommand{\E}{\bb{E}}

\begin{document}


\thispagestyle{empty}

\begin{center}
{\LARGE \bf   Empirical limit theorems for Wiener  chaos  }\\
\end{center}
\begin{center}
    Shuyang Bai, Jiemiao Chen\footnote{The authors are ordered alphabetically and contributed equally to this work.}
\end{center}
\begin{abstract}
{
We consider empirical measures in a triangular array setup with underlying distributions varying as sample size grows. We study asymptotic properties of multiple integrals with respect to normalized empirical measures. Limit theorems involving series of multiple Wiener-It\^o integrals are established.
}


\end{abstract}

Mathematics Subject Classification (2020) 60F05 (1st); 60H05 (2nd)
{
\section{Introduction}\label{s0}

%
Let $\{X_j\}$ be independent and identically distributed (i.i.d.) random elements taking value in a measurable space $(E,\cl{E})$ with distribution $P_0$.  The \emph{empirical measure}  $\widehat{P}_n=\frac{1}{n}\sum_{j=1}^n \delta_{X_j}$, where $\delta_{(\cdot)} $  denotes the Dirac delta measure and $n\in \bb{N}_+=\{1,2,\ldots\}$ is the sample size, is a well-studied object in probability and statistics  (see, e.g., \cite{vaart1997weak}).    In particular,  it is well-known that the normalized set-indexed empirical process 
$
\pp{\sqrt{n} \pp {\widehat{P}_n-P_0}(B)}_{B\in \cl{E}} ,$ 
 converges in finite-dimensional distributions to a set-indexed  zero-mean Gaussian process $\pp{G(B)}_{B\in \cl{E}}$ with covariance
 \begin{equation}\label{eq:Brown bridge cov}
 \E G(B_1)G(B_2)= P_0(B_1\cap B_2)-P_0(B_1)P_0(B_2),\quad  B_1,B_2\in \cl{E}. 
 \end{equation} Observe that $G(B_2)$ and $G(B_1)$, for $B_1,B_2\in \cl{E}$ with positive $P_0$ measures are negatively correlated  if $B_1\cap B_2=\emptyset$.  We shall refer to $G$ as a \emph{Brownian bridge random measure}.

On the other hand,  another well-known class of set-indexed Gaussian processes is the so-called \emph{Brownian random measure}. Suppose $\mu$ is a (possibly infinite) measure on $(E,\cl{E})$, which we refer to as the \emph{control measure}. Then a Brownian random measure with control measure $\mu$ is a set-indexed zero-mean Gaussian process $\pp{W(B)}_{B\in \cl{E}_0}$, 
$
\mathcal{E}_0=\{B \in \mathcal{E}: \mu(B)<\infty\},
$
with covariance $\E W(B_1)W(B_2)= \mu(B_1\cap B_2)$, $B_1,B_2\in \cl{E}_0$. Unlike the Brownian bridge random measure $G$, the Brownian random measure $W$ is \emph{independently scattered}: if $B_1, \ldots, B_n \in \mathcal{E}_0$ are pairwise disjoint, then $W\left(B_1\right), \ldots, W\left(B_n\right)$ are independent. Brownian random measure can be viewed as a way to make mathematical sense of the notion of Gaussian white noise and plays important roles in stochastic analysis (e.g., \cite{nualart2006malliavin}) and beyond. We note that although the term ``random measure'' is used to address $G$ and $W$, generally neither of them admits realizations  almost surely as  signed measures of finite variations (see, e.g., \cite{bryc1995gaussian}).

Several articles \cite{ivanoff1982central,ossiander1985levy,alexander1986uniform,bierme2014invariance}   have considered limit theorems that start with an empirical-like random measure and end up with a Brownian random measure $W$ as the limit,
although none of them considered exactly empirical measures. As already mentioned above, a properly normalized $\widehat{P}_n$ leads to a Brownian bridge random measure $G$ with covariance in \eqref{eq:Brown bridge cov} that typically does not vanish for disjoint subsets $B_1$ and $B_2$. In this paper, however, we shall consider a triangular array scheme: let the distribution $P_0$ be replaced by a sequence of changing distributions $P_n$, $n\in \bb{N}_+$, and impose a certain condition (see Assumption \ref{asu2} below), which ensures that $P_n(B_1)P_n(B_2)$ decays to zero at a faster rate than $P_n(B_1\cap B_2)$. This creates an asymptotic uncorrelatedness of the empirical measures evaluated at disjoint $B_1$ and $B_2$.  Under such a scheme, a properly normalized empirical measure, denoted $W_n$, will converge in finite-dimensional distributions to a Brownian random measure $W$.

The main goal of this paper is to study multiple empirical integrals, each of the form:
 \begin{equation}\label{eq I_n}
I_k^{(n)}(f): = \int_{E^k}^{\prime} f\left(u_1, \ldots, u_k\right) W_n\left(d u_1\right) \ldots W_n\left(d u_k\right), \quad k\in \bb{N}_0: = \{0,1,2,\ldots \},
\end{equation}
for suitable integrand $f$, where the prime $'$ means the exclusion of the diagonals $u_i=u_j$, $i\neq j$, from the multiple integral.  Such types of multiple empirical integrals have been considered by Major \cite{major2005tail, major2005tail1, major2005estimate,major2006estimate, major2013estimation} and Boistard and del Barrio \cite{boistard2009central} with a fixed distribution $P_n\equiv P_0$.    
Our results show that as $n\rightarrow\infty$, the multiple empirical integral $I_k^{(n)}(f)$ defined by our $W_n$ with changing distributions $P_n$ converges weakly to a corresponding multiple integral with respect to a Brownian random measure $W$, that is, a \emph{multiple Wiener-It\^o integral} (\cite{ito1951multiple,major2014multiple}; see also Section \ref{msi} below) denoted as $I_k(f)$.  Limit theorems or approximations for multiple stochastic integrals in similar spirits have also been considered in \cite{budhiraja1997two,jeon2004wong,tudor2007approximation,bardina2009convergence,bardina2010approximation,shen2012convergence,sun2017weak}.

More generally, our main results, Theorem \ref{thm6} and Proposition \ref{prop4}, concern limit theorems for $\sum_{k=0}^{K_n} I_k^{(n)}(h^{(k)})$ with suitable integrands $h^{(k)}$ defined on $E^k$, $k\in \bb{N}_0$, where $K_n$ is an integer sequence that grows to $\infty$ as $n\rightarrow\infty$. 
The limit is a series of multiple Wiener-It\^o integrals $\sum_{k=0}^{\infty} I_k(h^{(k)})$, known as a \emph{Wiener chaos} expansion. A key step to establish these limit theorems involving the growing truncation $K_n$ is an asymptotic analysis of the second-order moment properties of the multiple empirical integrals, which may be of independent interest; see Proposition \ref{prop2} and Section \ref{sec:pf prop}.
Limit theorems involving a  growing truncation and an infinite series of multiple Wiener-It\^o integral integrals have also appeared in the study of degenerate U-statistics. See, e.g.,  \cite{dynkin1983symmetric, mandelbaum1984invariance}.  The mechanism by which a Brownian random measure arises in these U-statistics limit theorems is, however, different: it is due to a cancelation effect from a degeneracy condition on the integrand. 
See \cite{boistard2009central} for clarification.

The rest of paper is organized as follows. We introduce the preliminaries in Section \ref{s2}, present the main results in Section \ref{s3}, and provide proofs in Section \ref{s4}.

\section{Preliminaries}\label{s2}
\subsection{Notations and basic assumptions}\label{s1}
Throughout the paper, $|A|$ denotes the cardinality of a set $A$. 
Suppose $a_n, b_n$ are two real-valued sequences indexed by $n\in \bb{N}_+$. Then $a_n \sim b_n$  means $\lim_{n\rightarrow \infty}a_n/ b_n = 1$, $a_n =  o\left(b_n\right)$ means that $\lim_{n\rightarrow\infty} a_n/b_n = 0$, $a_n = O\left(b_n\right)$ means there exists a constant $C>0$ which does not depend on $n$,   such that $|a_n|\leq C |b_n|$. Also, $o_n(1)$ stands for a sequence that tends to zero as $n\rightarrow\infty$,  $a_n\nearrow a$ if $a_n\le a_{n+1}$ and $\lim_n a_n=a$ with $a$ being a constant. For a measure $m$ and  a non-negative or integrable function $f$, $m(f)$ denotes the integral $\int f dm$,  and $L^p(m)$, $p>0$, denotes the $L^p$ space of functions with respect to the measure $m$. Write $m^k$, $k\in \bb{N}_+$, the product measure on the $k$-product space formed by the original space on which $m$ is defined. In particular, for a function $g$ defined on the product space, $m^k(g)$ stands for the integral with respect to $m^k$ (instead of the integral value $m(g)$ to the power $k$). $\bb{P}$ denotes the  probability measure of the underlying probability space. The arrow $\Rightarrow$ stands for weak convergence, and $\overset{d}{=}$ stands for equality in distribution. Besides, we let $C$, $C_i$'s, $i=0,1,2$ be generic positive constants whose value may change from expression to expression.

Now we introduce the setup we shall follow in the rest of the paper. Suppose  for each $n \in \bb{N}_+$  that $\left\{X_{n, j},\ j \in \bb{N}_+ \right\}$ are  i.i.d.\ random elements taking value in a measurable space $(E, \mathcal{E})$. We denote the common distribution of $X_{n,j}$ by $P_n$ for fixed $n$, i.e. $P_n( \, \cdot\, )= \bb{P}\left(X_{n, 1} \in \cdot\,\right)$; thus each $P_n$ is a probability measure on $(E,\cl{E})$. 
 Let $\mathcal{E}_0$ be the sub-collection of $\mathcal{E}$ defined by $\mathcal{E}_0=\{B \in \mathcal{E}: \mu(B)<\infty\}.$ 
We first introduce the assumption that will be used later as follows. 
\begin{asu}\label{asu2} Assume
    $(E,\mathcal{E}, \mu)$ is a non-atomic $\sigma$-finite measure space with  $\mu(E)=\infty$. Suppose a positive sequence $a_n \rightarrow \infty$, $a_n=o(n)$, as $n\rightarrow\infty$. Write $\mu_n =  (n/a_n)P_n$.  Assume that $\mu_n \leq \mu_{n+1}$ and $\lim_{n \rightarrow \infty} \mu_n = \mu$, where the two relations hold setwise on $\mathcal{E}$.
\end{asu}
\begin{rem}
Assumption \ref{asu2} implies that $P_n(B)\rightarrow 0$ as $n\rightarrow\infty$ for any $B\in \mathcal{E}_0$.
\end{rem}
We provide below a construction of $P_n$ satisfying Assumption \ref{asu2}.
\begin{exam}\label{exam1}
 Suppose a sequence of non-decreasing subsets $E_n\in \cl{E}_0$  such that $\mu(E_n)>0$,  $\bigcup_{n\ge 1} E_n = E$ and $\mu(E_n)=o(n)$ as $n\rightarrow\infty$.  Note that the last requirement is always possible since $\mu$ is non-atomic.  Then set $ a_n = \frac{n}{\mu\left(E_n\right)}$.   Define probability measures $P_n$ on $(E,\cl{E})$ by $P_n(B)  = \frac{\mu(E_n \cap B)}{\mu(E_n)}= \frac{a_n}{n} \mu(E_n \cap B) $, $B\in \cl{E}$.
\end{exam}

\subsection{Multiple stochastic integrals}\label{msi}

First, we briefly review multiple Wiener-It\^o integrals,  and refer  to  \cite[Section 1.1.2]{nualart2006malliavin} and \cite[Appendix B]{pipiras2017long} for more details.
Suppose $W$ is a  Brownian random measure on $(E, \mathcal{E})$ with non-atomic control measure $\mu$. A so-called  multiple Wiener-It\^o integral with respect to $W$ is  written as
\begin{equation}\label{eq35}
 \int_{E^k}^{\prime} f\left(u_1, \ldots, u_k\right) W\left(d u_1\right) \ldots W\left(d u_k\right)=: I_k(f),   
\end{equation}
where $f\in L^2(\mu^k)$, $k\in \bb{N}_0$.  In particular, we understand a function without variables, that is, $f\in L^2(\mu^0)$, as a finite constant, say $c\in \bb{R}$, for which we set $I_0(f)=c$. We also understand the notation $\mu^0(f)=c$ in this case. The prime in $\int_{E^k}^{\prime}$ refers to the fact that integration excludes the diagonals $\{u_i= u_j,\ i \neq j\}$.

For $k\ge 1$, the integral (\ref{eq35}) is first defined for the class of simple functions on $E^k$ denoted as $\mathcal{S}^k$. Here, a function $f$: $E^k \mapsto \mathbb{R}$ is said to be simple, i.e., $f\in \cl{S}^k$, if it has the form
\begin{equation}\label{simpdef}
f\left(u_1, \ldots, u_k\right)=\sum_{i_1, \ldots, i_k=1}^t a_{i_1, \ldots, i_k} \mathbb{I}_{A_{i_1} \times \ldots \times A_{i_k}}\left(u_1, \ldots, u_k\right),
\end{equation}
where $A_1, \ldots, A_t \in \mathcal{E}_0$, $t\in \bb{N}_+$, are pairwise disjoint, and  $a_{i_1, \ldots, i_k}$ are real-valued coefficients whose value is  zero if any two of the indices $i_1, \ldots, i_k$ are equal.  The last property ensures that the function $f$ vanishes on the diagonals $\{u_i=u_j:\ i \neq j\}$. For such $f$, the multiple integral $I_k(f)$ is defined by replacing the indicator in \eqref{simpdef} with the product $W(A_{i_1})\ldots W(A_{i_k})$.  
The definition of $I_k(f)$ can  be then extended to a function  $f \in L^2\left( \mu^k\right)$ through an $L^2$-approximation with simple functions.   For $k=0$,  $f\in \cl{S}^0$ is understood as a finite constant.

We mention some important properties of multiple Wiener-It\^o integrals.  For $f, g \in L^2(\mu^k)$, $k\in \bb{N}_0$ and $a, b \in \mathbb{R}$, we have almost surely
$
I_k(a f+b g)=a I_k(f)+b I_k(g),
$
and
$
I_k(f)=I_k(\wt{f}),
$
where $\wt{f}$ denotes the symmetrization of $f$ defined by
\begin{equation}\label{eq37}
 \wt{f}\left(u_1, \ldots, u_k\right)=\frac{1}{k !} \sum_\sigma f\left(u_{\sigma(1)}, \ldots, u_{\sigma(k)}\right),   
\end{equation}
with the sum over all permutations $\sigma$ of $\{1, \ldots, k\}$.   Moreover, for   $f\in L^2(\mu^{k_1})$, $g \in L^2(\mu^{k_2})$,  $k_1,k_2\in \bb{N}_0$, we have 
\begin{align}\label{eq38}
\begin{gathered}
\mathbb{E} I_{k_1}(f) I_{k_2}(g)=\left\{\begin{array}{cc}
0, & \text { if } k_1 \neq k_2; \\
k !\langle\widetilde{f}, \widetilde{g}\rangle_{L^2\left( \mu^k\right)}, & \text { if } k_1=k_2=:k ,
\end{array}\right.
\end{gathered}    
\end{align}
where $\langle\cdot , \cdot \rangle_{L^2\left( \mu^k\right)}$ stands for the $L^2$ inner product.
Note that  when $k_1=k\ge 1$ and $k_2=0$,  the relation  essentially says $\mathbb{E}I_{k}(f)= 0$. When $k_1=k_2=k$, and $f=g$ above, it follows from (\ref{eq37}) and a triangular inequality that
\begin{equation}\label{eq:mwi 2nd moment}
\mathbb{E} I_k(f)^2=k !\|\wt{f}\|_{L^2\left( \mu^k\right)}^2 \leq k !\|f\|_{L^2\left( \mu^k\right)}^2 .
\end{equation}
 Below we introduce a space  describing the integrands of Wiener chaos of possibly infinite order.
\begin{defn}\label{defnofH}
Let $H$ stand for the set of  $h=\left(  h^{(0)}, h^{(1)}\left(x_1\right), h^{(2)}\left(x_1,  x_2\right),...\right)$ where $h^{(k)} \in L^2(\mu^k)$, $k\in \bb{N}_0$, which satisfies
\begin{equation}\label{eq:H norm}
\|h\|_{H}^2:=\sum_{k=0}^{\infty} k! \|\wt{h}^{(k)}\|^2_{L^2\left( \mu^k\right)}<\infty,
\end{equation}
where $\wt{h}^{(k)}$ is the symmetrization of  $h^{(k)}$.
\end{defn}
For $h\in H$,   the  Wiener chaos series expansion $\sum_{k=0}^\infty I_k(h^{(k)})$ is well defined in $L^2(\mathbb{P})$.

Next,  we turn to the multiple empirical integral in \eqref{eq I_n}. With the sequence $a_n$ as in Assumption \ref{asu2}, we define the normalized empirical measure  $W_n:=\left(W_n(B)\right)_{B \in \mathcal{E}}$ as follows
\begin{equation}\label{eq3}
    W_n(B) = \frac{1}{\sqrt{a_n}} \sum_{i=1}^n \left[ \delta_{X_{n,i}}(B)-P_n(B)  \right].
\end{equation}
 Since each realization of $W_n$ in \eqref{eq3} is a signed measure with bounded total variation, 
 the  expression (\ref{eq I_n})     can be  realization-wise understood as a multiple integral (excluding the diagonals) with respect to the signed measure $W_n$  (cf., e.g., \cite[P.182]{bogachev2007measure}).  Under  $f \in L^1(P_n^k)$,  one can verify that the integrability of the multiple integral (\ref{eq I_n})    holds almost surely.  
Note that if $f\in L^2(\mu^k)$ and $(n/a_n)P_n \leq \mu$ as imposed in Assumption \ref{asu2}, we deduce $f \in L^2(P_n^k)$, which further implies $f \in L^1(P_n^k)$.  We follow a similar convention as above to understand the integrals of the zeroth order $I_0^{(n)}(f)$ and $P_n^{0}(f)$ as a constant.

\section{Main results}\label{s3}

 First, we present a result that can be viewed as an asymptotic version of the relation \eqref{eq38}.

\begin{prop}\label{prop2}
 For $f \in L^2\left(  \mu^{k_1}\right)$, $g \in L^2\left( \mu^{k_2}\right)$,   $k_1, k_2\in \bb{N}_0$, we have as $n \rightarrow \infty$,
\begin{align}
\begin{gathered}
\mathbb{E}\left(I_{k_1}^{(n)}\left(f\right) I_{k_2}^{(n)}\left(g\right)\right) \rightarrow \left\{\begin{array}{cc}
0, & \text { if } k_1 \neq k_2; \\
k !\langle\widetilde{f}, \widetilde{g}\rangle_{L^2\left( \mu^k\right)}, & \text { if } k_1=k_2=:k .
\end{array}\right.
\end{gathered}    
\end{align}

\end{prop}
The proof of Proposition \ref{prop2} can be found in Section \ref{sec:pf prop} below.

The next result concerns weak convergence towards infinite-order Wiener chaos.
\begin{thm}\label{thm7}
    Suppose $h  = \{ h^{(k)}\}_{k=0}^\infty \in H$ . Then there exists a sequence $K_n \nearrow \infty$ such that as $n\rightarrow \infty$,
    \begin{equation}\label{limit thm}
        \sum_{k=0}^{K_n} I_k^{(n)}(h^{(k)}) \Rightarrow \sum_{k=0}^\infty I_k(h^{(k)}).
    \end{equation}
\end{thm}
See Section \ref{sec:pf limit} below for the proof of Theorem \ref{thm7}, and see Section \ref{s0} for a discussion of its relation to existing literature.   Notice that in principle the sequence $K_n$  in Theorem \ref{thm7} could tend to infinity arbitrarily slowly. Under a specific setup, we provide an explicit rate of $N_n$. 

\begin{rem}
    As suggested by an anonymous reviewer, it is of interest to consider $\{X_{n,i}\}$ that is dependent along $i$. When the dependence is weak, we expect that the conclusions of Proposition \ref{prop2} and Theorem \ref{thm7} still hold, with a possible modification involving a multiplicative constant factor in the limits; see e.g, \cite{dehling2002empirical}. When the dependence is strong, i.e., when in the regime of \emph{long-range dependence} \cite{pipiras2017long}, a new type of limit behaviors may emerge; see, e.g.,   \cite{dehling1989empirical}. We leave the exploration of these scenarios to future works.
\end{rem}

\begin{prop} \label{prop4}
Assume the setup in Example \ref{exam1}, where $P_n(\cdot ) = \mu(\cdot \cap E_n)/\mu(E_n)$ and $a_n=\mu(E_n)$. In addition, for $h \in H$ as in Definition \ref{defnofH}, suppose that
    there exists $n_0\in \bb{N}_+$  such that  the support     $\operatorname{supp}(h^{(k)}) \subset E_{n_0}^{k}$ for all $k\in \bb{N}_+$. 
Then the conclusion of Theorem \ref{thm7} holds with     $K_n = O\left( \left[\ln \pp{n/a_n}\right]^{1-\varepsilon} \right)$ as $n\rightarrow\infty$ for any $0<\varepsilon<1$.
\end{prop}
The proof of Proposition \ref{prop4} is included in  Section \ref{sec:pf limit} below.

\begin{exam}
 To construct a concrete example to illustrate Theorem \ref{thm7} and Proposition \ref{prop4},  take $(E,\cl{E},\mu)=([0,\infty),\mathfrak{B}([0,\infty)),\lambda)$, where $\lambda$ denotes the Lebesgue measure. Set $E_n=[0,n^{1/2}]$, $a_n=n^{1/2}$, and define $P_n(\cdot )= n^{-1/2} \lambda(\cdot \cap E_n) $.  Then $\mu_n(\cdot)= \lambda(\cdot \cap E_n)\nearrow \lambda(\cdot)$  as $n\rightarrow\infty$. Further, let $h^{(k)} = k^{-1}\mathbb{I}_{[0,1]\times[0,1/2]\times \cdots \times [0,1/k]}$, and relation (\ref{eq:H norm}) can be readily verified. Let $n_0 = 1$, $\varepsilon = 1/2$, and with $K_n = O\left( \left[\ln  n \right]^{1/2} \right)$, the conclusion (\ref{limit thm}) follows. 
\end{exam}

\section{Proofs}\label{s4}

\subsection{Asymptotic moment properties of multiple empirical integrals} \label{sec:pf prop}

We first prepare a lemma regarding a sequence of non-decreasing measures.
\begin{lem}\label{lem2}
Suppose $(U,\mathcal{G},m)$ is a measure space. Assume  a sequence of measures $m_n$ on $(U,\cl{G})$ satisfies $m_n\nearrow m$ setwise on $\mathcal{G}$ as $n\rightarrow\infty$. Then  for any non-negative measurable function $f$ on $U$, we have  $m_n(f) \nearrow m (f)$ as $n\rightarrow\infty$.
   For any $m$-integrable  function $f$ on $U$, we have 
$m_n(f) \rightarrow m(f)$  as $n\rightarrow\infty$. 
\end{lem}
\begin{proof}
For a non-negative measurable $f$, there exists a sequence of non-negative simple functions $ g_\ell \nearrow f$ pointwise as $\ell\rightarrow\infty$. It follows from the assumption and linearity that $m_n(g_\ell) \nearrow m(g_\ell)$ as $n \rightarrow \infty$. Furthermore, $m(g_{\ell}) \nearrow m(f)$ as $\ell \rightarrow \infty$.   The claimed monotonicity follows from letting $\ell\rightarrow\infty$ in $m_n(g_\ell)\le m_{n+1}(g_\ell)$.   Since $m_n(g_\ell)$ is non-decreasing with respect to  either of the indices $n$ and $\ell$, we can interchange the limit as $\lim_{\ell }\lim_{n} m_n(g_\ell) = \lim_{n} \lim_{\ell}m_n(g_\ell)$. This implies the first convergence.
For $f \in L^1(m)$, write $f=f^+-f^-$ where $f^{+}=\max \{0, f\}, f^{-}=\max \{0,-f\}$. Then applying  the first claim to both $f^+$ and $f^-$ yields  the second claim.
\end{proof}

{Below following \cite{major2005estimate}, we introduce certain combinatorial notions useful for analysis of moments of multiple empirical integrals.

 We define the \emph{diagram} $B(\mathcal{N})=B\left(\mathcal{N}, k_1, k_2\right)$ as a diagram with $k_1$ vertices in the first row and $k_2$ vertices in the second row with edge set
 \begin{equation}\label{eq:edge set}
 \mathcal{N}=\pc{\left(j_1, j_1^{\prime}\right), \ldots, \left(j_l, j_l^{\prime}\right)},
 \end{equation} 
 where $1 \leq j_s \leq k_1$, $k_1+1 \leq j_s^{\prime} \leq k_1+k_2$, and $j_s \neq j_{s^{\prime}}, j_s^{\prime} \neq j_{s^{\prime}}^{\prime}$ if $s \neq s^{\prime}$, $1\le s,s'\le \ell$ and $\ell \leq \min\{ k_1, k_2\}$. Then given a diagram $B(\mathcal{N})$ and a subset $\mathcal{N}_1$ of $\mathcal{N}$, we define the \emph{colored diagram} $B\left(\mathcal{N}, \mathcal{N}_1\right)=B\left(\mathcal{N}, \mathcal{N}_1, k_1, k_2\right)$ by coloring the edge set $\cl{N}$ of a diagram $B\left(\mathcal{N}\right)$  by two different colors depending on whether or not an edge belongs to $\mathcal{N}_1$.

Given two real-valued measurable functions $f\left(x_1, \ldots, x_{k_1}\right)$ and $g\left(y_1, \ldots, y_{k_2}\right)$ on the spaces $\left(E^{k_1}, \mathcal{E}^{k_1}\right)$ and $\left(E^{k_2}, \mathcal{E}^{k_2}\right)$ respectively, we   define the tensor product function as
\begin{equation}\label{eq:f g tensor}
f \otimes g\left(x_1, \ldots, x_{k_1+k_2}\right)=f\left(x_1, \ldots, x_{k_1}\right) g\left(x_{k_1+1}, \ldots, x_{k_1+k_2}\right).
\end{equation}
Besides, denote $\left(f \otimes g\right)_{B(\mathcal{N})}$ the function of $k_1+k_2-l$ variables obtained by identifying the   pairs of arguments of $f \otimes g$ according to $\cl{N}$, i.e., setting    $x_{j_s} =x_{j_s^\prime}$  in the function $f \otimes g\left(x_1, \ldots, x_{k_1+k_2}\right)$, $s = 1,2, \ldots, l$. For example, if $k_1=3$, $k_2=2$, $\cl{N}=\{(1,4),(3,5)\}$, then $\left(f \otimes g\right)_{B(\mathcal{N})}(y_1,y_2,y_3)= f(y_1,y_2,y_3)g(y_1,y_3)$.  When $\cl{N}=\emptyset$ (or equivalently $l=0$), we understand $\left(f \otimes g\right)_{B(\mathcal{N})} $ as $ f \otimes g $.  

Next denote $\mathcal{B}(l) = \mathcal{B}(l,k_1,k_2)$ the collection of  all the diagrams $B(\mathcal{N})$ with $|\mathcal{N}|=l$ edges. We define     
\begin{equation}\label{abn}
    \overline{(f \otimes g)}_{\mathcal{B}(l)} = \frac{1}{|\mathcal{B}(l,k_1,k_2)|}\sum_{B(\mathcal{N})\in \mathcal{B}(l,k_1,k_2) } (f \otimes g)_{B(\mathcal{N})},
\end{equation}
where $|\mathcal{B}(l,k_1,k_2)|=\frac{k_{1} ! k_{2} !}{\left(k_1-l\right) !\left(k_2-l\right) ! l !}$. Note that when $l = 0$, we have $\overline{(f \otimes g)}_{\mathcal{B}(l)} = f\otimes g$.

Now suppose $B\left(\mathcal{N}, \mathcal{N}_1\right)=B\left(\mathcal{N}, \mathcal{N}_1, k_1, k_2\right)$ is   a colored diagram with $|\cl{N}|=l$ and $|\cl{N}_1|=p$.
We use
$
\left(f \otimes g\right)_{B\left(\mathcal{N}, \mathcal{N}_1\right)} 
$
to denote the function of $k_1+k_2-l-p$ variables by integrating $(f \otimes g)_{B(\mathcal{N})}$ out the subset of arguments corresponding to the edges in $\cl{N}_1$  with respect to $P_n$,  given that integrability holds.  For example, if $k_1=3$, $k_2=2$, $\cl{N}=\{(1,4),(3,5)\}$, $\cl{N}_1=\{(1,4)\}$, then $\left(f \otimes g\right)_{B(\mathcal{N},\mathcal{N}_1)}(y_1,y_2)= \int_E f(x,y_1,y_2)g(x,y_2) P_n(dx)$. When $\cl{N}=\cl{N}_1=\emptyset$ (or equivalently $l=p=0$), we understand    
$\left(f \otimes g\right)_{B(\mathcal{N},\cl{N}_1)}$  as $f \otimes g$.

We  use  $\mathcal{B}(l, p)=\mathcal{B}\left(l, p, k_1, k_2\right)$ to denote the collection of colored diagrams $B\left(\mathcal{N}, \mathcal{N}_1\right)=B\left(\mathcal{N}, \mathcal{N}_1, k_1, k_2\right)$ with $k_1$ vertices in the first row and $k_2$ vertices in the second row,   $|\mathcal{N}|=l$, and $\left|\mathcal{N}_1\right|=p$. We then define
\begin{align}\label{defBlp}
    \overline{(f \otimes g)}_{\mathcal{B}(l,p)} &=  \overline{(f \otimes g)}_{\mathcal{B}(l,p,k_1,k_2)} = \frac{1}{|\mathcal{B}(l,p,k_1,k_2)|} \sum_{B(\mathcal{N},\mathcal{N}_1) \in \mathcal{B}(l, p,k_1,k_2)} (f \otimes g)_{B\left(\mathcal{N}, \mathcal{N}_1\right)},
\end{align}
where $|\mathcal{B}(l,p,k_1,k_2)| = \frac{k_{1} ! k_{2} !}{\left(k_1-l\right) !\left(k_2-l\right) !(l-p) ! p !}$. Note that when $l = 0$, we have $\overline{(f \otimes g)}_{\mathcal{B}\left(l, p, k_1, k_2\right)} = f\otimes g$.  
}

We shall also take advantage of some results from \cite{major2005estimate}.
The following   inequality, which essentially follows from a Cauchy-Schwartz inequality, will be useful:  for $f\otimes g$ as in \eqref{eq:f g tensor}  and any edge set $\cl{N}$ of the form \eqref{eq:edge set}, we have
\begin{equation}\label{contraction2}
\left\|\left(f \otimes g\right)_{B(\mathcal{N})}\right\|_{L^2 \left(P_n^{k_1+k_2-l}\right)}=\left\|\left(f \otimes g\right)_{B(\mathcal{N},\mathcal{N})}\right\|_{L^2 \left(P_n^{k_1+k_2-2l}\right)} \leqslant \|f\|_{L^2 (P_n^{k_1})} \cdot\|g\|_{L^2(P_n^{k_2})}\text{.}
\end{equation}
See \cite[Equation (3.15)]{major2005estimate}.  We note that \eqref{contraction2} continues to hold if $P_n$ is replaced by a $\sigma$-finite measure.

The following lemma is \cite[Lemma 1]{major2005estimate}  adapted to our notation. In particular, the notation  $J_{n, k}(f)$ in \cite{major2005estimate} follows the relation $J_{n, k}(f) = \frac{1}{k!} (\frac{a_n}{n})^{k/2} I^{(n)}_k(f)$, where  $I^{(n)}_k(f)$ is as defined in (\ref{eq I_n}).

\begin{lem}\label{lem4}
 Suppose $f\in L^1(\mu^k)$, $k\in \bb{N}_0$.   Then for $n\in \bb{N}_+$,
$$
E I^{(n)}_k(f)=k!\, B_{n, k}\,\left(\frac{n}{a_n}\right)^{k/2}  P_n^{ k}(f),
$$
where for $k\ge 1$,
$$
B_{n, k}=\frac{1}{k ! n^{k / 2}} \sum_{s=1}^k(-1)^{k-s}\left(\begin{array}{c}
n \\
s
\end{array}\right) s ! \times\left(\sum_{\substack{\left(r_1, \ldots, r_s\right) \\
r_j \in \mathbb{N}_{+}, 1 \leq j \leq s, r_1+\cdots+r_s=k}}\left(r_1-1\right) \cdots\left(r_s-1\right) B\left(r_1, \ldots, r_s\right)\right),
$$
and $B\left(r_1, \ldots, r_s\right)$ equals the number of partitions of the set $\{1, \ldots, k\}$ into disjoint sets with cardinalities $r_1, \ldots, r_s\ge 1$, and we set $B_{n,0}=1$.
 The coefficient $B_{n, k}$ satisfies the estimate (cf.\ \cite[Equation (3.10)]{major2005estimate})
\begin{equation}\label{eq:B_n,k bound}
\left|B_{n, k}\right| \leq \frac{C^k}{k^{k / 2}}
\end{equation}
with some constant $C>0$ which does not depend on $k$ and $n$, for all $n \geq \frac{k}{2}$.
\end{lem}

\begin{lem}\label{lem5}
   For any $f \in L^2( \mu^k)$, and fixed $k\in \bb{N}_+$, as $n \rightarrow \infty$,
    $
  \left(n/a_n\right)^{k/2} P_n^{ k}(f) 
  \rightarrow 0.
  $
\end{lem}
\begin{proof}
By Assumption \ref{asu2}, for any $B\in \cl{E}_0$, we have    $(n/a_n) P_n(B) \rightarrow \mu(B)$ as $n\rightarrow\infty$. Since $a_n=o(n)$, we deduce $ (n/a_n)^{1/2} P_n(B)\rightarrow 0$ as $n\rightarrow\infty$. 
   Using this fact and the linearity of the integral $P_n^k(\cdot)$, it can be readily verified that the conclusion of the lemma holds true for any simple function $g \in \mathcal{S}^k$.
   
    Now suppose $f\in L^2\pp{\mu^k}$. Take simple functions $ g_{m} \in \mathcal{S}^k$ such that $\| g_m - f\|_{L^2(\mu^k)} \rightarrow 0$ as $m \rightarrow \infty$.  Then
\begin{align*}
\left|\left(\frac{n}{a_n}\right)^{k / 2} P_n^{ k}(f) \,\right| & \leq \left| \left(\frac{n}{a_n}\right)^{k / 2} P_n^{ k}(f-g_m) \,\right| + \left|\left(\frac{n}{a_n}\right)^{k / 2} P_n^{ k}(g_m) \,\right|\\
& \leq  \left(\frac{n}{a_n}\right)^{k / 2} P_n^{ k}(\left|f-g_m\right|) + o_n(1) 
  \leq  \sqrt{\left(\frac{n}{a_n}\right)^k P_n^{ k}(\left|f-g_m\right|^2)} +  o_n(1)\\
& \leq \|f-g_m \|_{L^2(\mu^k)}+  o_n(1).
\end{align*}
For the last inequality above, recall $(n/a_n)P_n\le \mu$ from Assumption \ref{asu2}, which implies $(n/a_n)^k P_n^k\le \mu^k$.  
The conclusion follows then by letting $n \rightarrow \infty$  and then $m \rightarrow \infty$ at both ends of the inequality displayed above.
\end{proof}

The following lemma  is essentially \cite[Lemma 2]{major2005estimate} with an adaptation of the notation.
\begin{lem}\label{lem6}
(Diagram formula). Suppose $f \in L^2( \mu^{k_1})$, $g \in L^2( \mu^{k_2})$, $k_1,k_2\in \bb{N}_+$.  Then we have the following identity.
\begin{equation}\label{eq:prod decomp}
    I_{k_1}^{(n)}(f) I_{k_2}^{(n)}(g)=\sum_{l=0}^{\min \left(k_1, k_2\right)}  \sum_{p=0}^l \left(\frac{n}{a_n}\right)^{\frac{l+p}{2}} \left|\mathcal{B}(l, p)\right| n^{-\frac{l-p}{2}} I_{k_1+k_2-l-p}^{(n)}\left( \overline{(f \otimes g)}_{\mathcal{B}\left(l, p\right)} \right)
\end{equation}
with $\overline{(f \otimes g)}_{\mathcal{B}\left(l, p\right)}$  as in (\ref{defBlp}), $0 \leq p \leq l \leq$ $\min \left(k_1, k_2\right)$.
\end{lem}

To analyze the expectation of each term in  \eqref{eq:prod decomp} using Lemma \ref{lem4}, for $f\in L^2(\mu^{k_1})$ and $g\in L^2(\mu^{k_2})$, $0\le l\le \min(k_1,k_2)$, we define the following bilinear form:
\begin{align}\label{Fnl}
 F^{(n)}_l(f,g)   : =&    \left(\frac{n}{a_n}\right)^{(k_1+k_2) / 2} P_n^{ k_1+k_2-l} \left(\overline{(f \otimes g)}_{\mathcal{B}(l)}\right)
 =&\left(\frac{n}{a_n}\right)^{(k_1+k_2) / 2} P_n^{ k_1+k_2-l-p} \left(\overline{(f \otimes g)}_{\mathcal{B}(l,p)}\right),
\end{align}
where $0\le p\le l$.  Notice that $ F^{(n)}_l(f,g)= F^{(n)}_l(\wt{f},\wt{g})$ because the integration eliminates the effect of variable permutations in $f$ and $g$. 
The following lemma shows that this bilinear form is uniformly bounded in terms of $n$.
{
\begin{lem}\label{Lem:Fnl bounded}
Suppose $f \in L^2( \mu^{k_1})$, $g \in L^2( \mu^{k_2})$, $k_1,k_2\in \bb{N}_+$,  and $0\le l\le \min(k_1,k_2)$.  Then
$ 
    |F_l^{(n)}(f, g)| \leq \| f\|_{L^2(\mu^{k_1})} \| g\|_{L^2(\mu^{k_2})}
$  for any $n\in \bb{N}_+$.
\end{lem}

\begin{proof}
We have by a triangular inequality that 
\begin{align}
     \ |F_l^{(n)}(f, g)|  &\le  \left(\frac{n}{a_n}\right)^{\left(k_1+k_2\right) / 2} P_n^{ k_1+k_2-2l}\left(\overline{(|f| \otimes |g|)}_{\mathcal{B}(l,l)}\right) \notag\\
    &= \left(\frac{n}{a_n}\right)^{\left(k_1+k_2\right) / 2} \frac{1}{\left|\mathcal{B}\left(l, l\right)\right|} \sum_{B\left(\mathcal{N}, \mathcal{N}\right) \in \mathcal{B}\left(l, l\right)} \| (|f| \otimes |g|)_{B\left(\mathcal{N}_{,}, \mathcal{N}\right)}\|_{L^1\left(P_n^{k_1+k_2-2 l}\right)} \label{cont1} \\
      &\leq \left(\frac{n}{a_n}\right)^{\left(k_1+k_2\right)/2}\|f\|_{L^2\left(P_n^{k_1}\right)} \cdot\|g\|_{L^2\left(P_n^{k_2}\right)} \leq \|f\|_{L^2\left(\mu^{k_1}\right)} \cdot\|g\|_{L^2\left(\mu^{k_2}\right)}. \label{cont3}
\end{align}
From (\ref{cont1}) to (\ref{cont3}), it follows from the fact that the $L^1$ norm is bounded by the $L^2$ norm  under a probability measure, as well as the relation (\ref{contraction2}). The last inequality follows from $(n/a_n) P_n \leq \mu$ imposed in Assumption \ref{asu2}.
\end{proof}

 \begin{lem}\label{lem13}
Suppose $f \in L^2( \mu^{k_1})$, $g \in L^2( \mu^{k_2})$, $k_1,k_2\in \bb{N}_+$.  
 Under either of these assumptions:  (i) $k_1 \neq k_2$, $0\leq l \leq \min \{k_1,k_2 \}$; or (ii)  $k_1 = k_2$ and $0\leq l <k_1$, we have as $n \rightarrow \infty$,
\begin{equation}\label{glem}
 F^{(n)}_l(f,g)  \rightarrow 0,
\end{equation}
where  $F^{(n)}_l(f,g)$ is as in \eqref{Fnl}.
\end{lem}

\begin{proof}~\\
\textbf{Step 1}. 
Suppose $f=\mathbb{I}_{A_{1, 1} \times \ldots \times A_{1, k_1}}$ and $g=\mathbb{I}_{A_{2, 1} \times \ldots \times A_{2, k_2}}$, where for each $i=1,2$, the collection $\{A_{i,j},\ 1\le j\le k_i\}$ consists of pairwise disjoint subsets in $\cl{E}_0$. We show that \eqref{glem} holds for such a pair $f$ and $g$.  We further divide the discussion into some substeps. 

    \emph{Step 1a}. Suppose $A_{1, 1}, \ldots, A_{1, k_1}, A_{2, 1}, \ldots, A_{2, k_2}$ are all pairwise disjoint. 
    
      When $l \neq 0$, we have the expression $(f \otimes g)_{B(\mathcal{N})}=0$ for any $B(\mathcal{N})\in \cl{B}(l)$  always.  This is clear once we observe that an identification of a pair of variables according to $\cl{N}$ results in the same variable inserted into two indicators of disjoint subsets.   Hence $ F^{(n)}_l(f,g)=0$ in this case.
    
%
    When $l =0$, we have $F_l^{(n)}(f, g) = \left(\frac{n}{a_n}\right)^{(k_1+k_2) / 2} P_n^{ (k_1+k_2)/2}(f \otimes g) \rightarrow 0$ as $n\rightarrow\infty$ by Lemma \ref{lem5}.
    

    \emph{Step 1b}. Now suppose that $\mathcal{G}_1:=\left\{A_{1, 1}, \ldots, A_{1, k_1}\right\}$ and $\mathcal{G}_2:=\left\{A_{2, 1}, \ldots, A_{2, k_2}\right\}$ are pairwise disjoint within each collection. Assume, in addition, that $\mathcal{G}_1$ and $\mathcal{G}_2$ share exactly $q=|\cl{G}_1\cap \cl{G}_2|$ common subsets, $1 \leq q  \leq \min \left\{k_1, k_2\right\}$, with  the subsets in $(\cl{G}_1\cup \cl{G}_2)\setminus (\cl{G}_1\cap \cl{G}_2)$ pairwise disjoint. In other words,  one has $\mathcal{G}_1 \cap \mathcal{G}_2 =:\{A_1, A_2, \ldots, A_q\}$, $\mathcal{G}_1 \cap\left(\mathcal{G}_2\right)^c = : \left\{B_{q+1}, \ldots, B_{k_1}\right\}$ and $ \mathcal{G}_2 \cap\left(\mathcal{G}_1\right)^c=:\left\{B_{k_1+1}, \ldots, B_{k_1+k_2-q}\right\}$, where   $A_i$'s $B_j$'s are all pairwise disjoint.

    If $0 \leq q < l$,  then $(f \otimes g)_{B(\mathcal{N})}=0$ for any $B(\mathcal{N})\in \cl{B}(l)$   as in \emph{Step 1a} above. Observe that when $ 0 \leq l \leq q$, there are at most $\binom{q}{l}$   non-zero $(f \otimes g)_{B(\mathcal{N})}$  as $B(\mathcal{N})$  varies in $\mathcal{B}(l)$.  We can then write (\ref{glem}) as 
$$
\begin{aligned}
 & \frac{\left( n /a_n\right)^{\left(k_1+k_2\right) / 2}}{|\mathcal{B}(l)|}\sum_{\left\{j_1, ., j_l\right\} \subset\{1,2, . ., q\}}  \pp{ \prod_{\alpha=1}^l P_n\left(A_{j_\alpha}\right)}\pp{ \prod_{\beta \in\{1,2, . ., q\} \setminus\left\{j_1, \ldots, j_l\right\}} P_n\left(A_\beta\right)^2} \ \times \\  & \pp{ \prod_{i=q+1}^{k_1+k_2-q} P_n\left(B_i\right)} 
  =  O\left(\left(\frac{a_n}{n}\right)^{\left(k_1+k_2\right) / 2-l} \right),  
\end{aligned}
$$
as $n\rightarrow\infty$, 
where the last asymptotic relation   holds since by Assumption  \ref{asu2}, we have  $P_n(B) = O(\frac{a_n}{n})$ for any $B \in \mathcal{E}_0$. Note that under either case (i) or (ii) in the lemma statement, we have $l<(k_1+k_2)/2$. Hence $ \left(\frac{a_n}{n}\right)^{\left(k_1+k_2\right) / 2-l}   \rightarrow 0$ as $n\rightarrow\infty$ in both cases, recalling $a_n=o(n)$ from Assumption \ref{asu2}. 

\emph{Step 1c}.  We  now return to the  general setup in Step 1:  the subsets $A_{1, 1}, \ldots A_{1, k_1}$ are disjoint and the same relation holds for $A_{2, 1}, \ldots, A_{2, k_2}$, whereas  $A_{1, i} \cap A_{2, j}$ may be nonempty.   After proper partitions of   these subsets, we can decompose $F^{(n)}_l(f,g)$, using the bilinearity of $F^{(n)}_l(\cdot ,\cdot)$, as a finite sum:
$$
F^{(n)}_l(f,g) = \sum_{ i,j   } F^{(n)}_l\left(f_i,g_j\right)
$$
 such that each pair   $f_i, g_j$   belongs to either the setup of \emph{Step 1a} or that of \emph{Step 1b}. We then apply the conclusions of these two cases to each term $F^{(n)}_l\left(f_i,g_j\right)$.

\noindent \textbf{Step 2}.   Suppose  now  $f \in \mathcal{S}^{k_1}$ and $g \in \mathcal{S}^{k_2}$. Following (\ref{simpdef}), we decompose $f$ and $g$ respectively as finite sums $f=\sum_{i} a_i f_{  i}$, $g=\sum_{j}  b_j g_{  j}$, where $a_i$ and $b_j$ are real coefficients, and each pair $f_{i}$, $g_{j}$ is of the form in Step 1.  
Then rewrite $F^{(n)}_l(f, g)$ as  a finite number of terms of the form $a_i b_j F^{(n)}_l(f_{ i},g_{ j})$,  each tending to zero as $n \rightarrow \infty$ by the conclusion of Step 1.

\noindent \textbf{Step 3}. Suppose now $f \in L^2\left( \mu^{k_1}\right)$ and $g \in L^2\left(\mu^{k_2}\right)$. Then there exist  sequences of functions $f_m\in \cl{S}^{k_1}$ and $g_m\in \cl{S}^{k_2}$, $m\in \bb{N}_+$, such that $\left\|f_m-f\right\|_{L^2(\mu^{k_1})} \rightarrow 0$ and $\left\|g_m-g\right\|_{L^2(\mu^{k_2})} \rightarrow 0$, as $m \rightarrow \infty$. The conclusion \eqref{glem} follows once we verify the following:
\begin{equation}\label{eq:stat 1}
\lim _{m \rightarrow \infty} \sup _{n\ge 1} \left|F^{(n)}_l(f,g) - F^{(n)}_l(f_m,g_m)\right|=0
\end{equation}
 and for fixed $m\in \bb{N}_+$, 
\begin{equation}\label{eq:stat 2}
\lim _{n \rightarrow \infty} F^{(n)}_l(f_m, g_m)=0.
\end{equation}
The relation \eqref{eq:stat 2} readily follows from the conclusion of Step 2. 
To show \eqref{eq:stat 1}, using a triangular inequality and Lemma \ref{Lem:Fnl bounded}, we have 
\begin{align}
\left| F^{(n)}_l(f,g) - F^{(n)}_l(f_m,g_m)\right| 
 \leq  &  \left|  F^{(n)}_l(f,g)- F^{(n)}(f,g_m)\right| + \left| F^{(n)}_l(f,g_m) - F^{(n)}_l(f_m,g_m) \right| \notag \\
 \leq  & \|f\|_{L^2( \mu^{k_1})}   \|g-g_m\|_{L^2( \mu^{k_2})} + \|f- f_m\|_{L^2( \mu^{k_1})}   \|g_m\|_{L^2( \mu^{k_2})}.\notag
\end{align}
The last bound does not depend on $n$ and  tends to $0$ as $m\rightarrow\infty$. Hence \eqref{eq:stat 1} holds and the proof is concluded.
\end{proof}
}

\begin{proof}[Proof of Proposition \ref{prop2}] ~  

\emph{Case 1.} When $k_1 \neq k_2$, either $k_1$ or $k_2$ equals zero (recall that a zeroth order integral is understood as a constant),  combining Lemma \ref{lem4} and Lemma \ref{lem5} yields the desired convergence to zero.

\emph{Case 2.} When $k_1 \neq k_2$ and both of them are non-zero, by Lemma  \ref{lem6}, a triangular inequality  and Lemma \ref{lem4}, we have
\begin{align}\label{eq:91}
    \left|\mathbb{E} I_{k_1}^{(n)}(f) I_{k_2}^{(n)}(g)\right|\le & \sum_{l=0}^{\min \left(k_1, k_2\right)} \sum_{p=0}^l\left(\frac{n}{a_n}\right)^{\frac{l+p}{2}} \left|\mathcal{B}\left(l, p \right)\right| n^{-\frac{l-p}{2}} \left|\mathbb{E} I_{k_1+k_2-l-p}^{(n)}\left(\overline{(f \otimes g)}_{\mathcal{B}\left(l, p \right)}\right)\right| \notag\\
     \leq & \sum_{l=0}^{\min \left(k_1, k_2\right)} \sum_{p=0}^l \left|\mathcal{B}\left(l, p \right)\right|  n^{-\frac{l-p}{2}}   \frac{C^{k_1+k_2-l-p} (k_1+k_2-l-p) !  }{(k_1+k_2-l-p)^{(k_1+k_2-l-p) / 2}} \notag\\
    &    \times  \left| \left(\frac{n}{a_n}\right)^{(k_1+k_2) / 2} P_n^{ k_1+k_2 -l -p}\left(\overline{(f \otimes g)}_{\mathcal{B}\left(l, p \right)}\right) \right|.
\end{align}
The expression inside the absolution value in (\ref{eq:91})  is $ F^{(n)}_l(f,g) $ in \eqref{Fnl}, which 
by   Lemma \ref{lem13}, tends to zero as $n\rightarrow
\infty$.

\emph{Case 3.} When $k_1 = k_2 = k$, we first show that as $n\rightarrow\infty$,
    \begin{equation}\label{eq:94}
    \left(\frac{n}{a_n}\right)^k P_n^{ k}\left(\sum_{B(\mathcal{N}) \in \mathcal{B}(k)}(f \otimes g)_{B(\mathcal{N})}\right) \rightarrow  k!\langle \widetilde{f}, \widetilde{g}  \rangle_{L^2\left( \mu^k\right)},
    \end{equation}
where $\wt{f}$ and $\wt{g}$ are symmetrizations defined as in (\ref{eq37}). To do so, observe that
\begin{align*}
   P_n^{ k} \left( \sum_{B(\mathcal{N}) \in \mathcal{B}(k)} (f \otimes g)_{B(\mathcal{N})}\right) &=   P_n^{ k} \left(f(x_1,\ldots,x_k)\sum_{\sigma   }g(x_{\sigma(1)},\ldots,x_{\sigma(k)})\right)\\
    & = P_n^{ k}\left( k! \wt{f}(x_1,..., x_k)\wt{g}(x_1,\ldots,x_k) \right),
\end{align*}
where the sum in the second expression above is over all permutations $\sigma$ of $\{1, \ldots, k\}$.
Note that $\wt{f}\wt{g} \in L^1 (\mu^k)$ due to Cauchy–Schwarz inequality $\| \wt{f}\wt{g}\|_{L^1(\mu^k)} \leq \| \wt{f}\|_{L^2(\mu^k)} \| \wt{g}\|_{L^2(\mu^k)} < \infty$.  Combining this with   Lemma \ref{lem2} yields (\ref{eq:94}).

Next, we have by Lemmas \ref{lem6} and \ref{lem4},
\begin{align}\label{eq:T0}
 \mathbb{E} I_{k}^{(n)}(f) I_{k}^{(n)}(g) & = \sum_{l=0}^{k} \sum_{p=0}^l\left(\frac{n}{a_n}\right)^{\frac{l+p}{2}} \left|\mathcal{B}\left(l, p\right)\right| n^{-\frac{l-p}{2}} \mathbb{E} I_{2k-l-p}^{(n)}\left(\overline{(f \otimes g)}_{\mathcal{B}\left(l, p\right)}\right) \notag\\
& = \sum_{l=0}^{k} \sum_{p=0}^l \, \left|\mathcal{B}(l, p )\right|\, n^{-\frac{l-p}{2}} (2k-l-p)! B_{n, 2k-l-p} \times  T(n,l),
\end{align}
where $T(n,l)=\left(\frac{n}{a_n}\right)^k P_n^{ 2k-l}\left(\overline{(f \otimes g)}_{\mathcal{B}(l)}\right)$.
 By Lemma \ref{lem13}, when  $l <k$, we have $T(n,l) \rightarrow 0$ as $n \rightarrow \infty$. Hence any term of the double summation \eqref{eq:T0} with $l<k$  tends to $0$. On the other hand, a term with $l =k$, $p<l$  also tends to zero. This is  because $n^{-\frac{l-p}{2}} \rightarrow 0$ as $n \rightarrow \infty$, and $T(n,k)$, which is exactly the left-hand side of \eqref{eq:94},  is bounded in $n$.
Therefore, only   the  term with $l= p = k$ contributes as  $n\rightarrow \infty$.  Note that $\left|\mathcal{B}(k, k )\right|=B_{n,0}=1$   in this case. So  in view of \eqref{eq:94} again, this contributing term converges to $k !\langle\widetilde{f}, \widetilde{g}\rangle_{L^2\left(\mu^k\right)}$.
\end{proof}

\subsection{Proofs of limit theorems for Wiener chaos}\label{sec:pf limit}

\begin{lem}\label{thm6}
    For any $h=\left(h^{(k)}\right)_{k=0}^\infty \in  H$, $K\in \bb{N}_0$ fixed, we have 

\begin{equation}\label{eq:joint}
\sum_{k=0}^{K} I_k^{(n)}(h^{(k)}) \Rightarrow \sum_{k=0}^{K} I_k(h^{(k)}).
\end{equation}

as $n\rightarrow\infty$.
\end{lem}

\begin{proof}
We first show that for any $B_i \in \mathcal{E}_0, i=1, \ldots m$, and $m \in \mathbb{N}_{+}$, as $n \rightarrow \infty$,
\begin{equation}\label{eq: conv for W}
\left(W_n\left(B_1\right), \ldots, W_n\left(B_m\right)\right) \Rightarrow\left(W\left(B_1\right), \ldots, W\left(B_m\right)\right).
\end{equation}
Indeed, by Cramér-Wold theorem, it is enough to show that as $n \rightarrow \infty$,
\begin{equation}\label{eq normal}
\frac{1}{\sqrt{a_n}} \sum_{j=1}^n\left[g\left(X_{n, j}\right)-P_n(g)\right] \Rightarrow N\left(0, \mu\left(g^2\right)\right) \stackrel{d}{=} \int_E g(x) W(d x)
\end{equation}
where $g(x)=\sum_{k=1}^n a_k \mathbb{I}_{B_k}(x)$,with $a_k \in \mathbb{R}$, $k=1, \ldots, n$.
If $g=0 \,\,\mu$-almost everywhere, the relation (\ref{eq normal}) holds trivially.
Now, we suppose $\mu\{g \neq 0\}>0$. Assumption \ref{asu2} implies that $P_n\left(g^2\right)\sim (a_n/n) \mu (g^2)$ and  $P_n(g)^2\sim (a_n/n)^2 \mu(g)^2$, and hence
the variance of the left-hand side of \eqref{eq normal}  is  $(n/a_n)\left(P_n\left(g^2\right)-P_n(g)^2\right) \rightarrow  \mu\left(g^2\right)$ as $n\rightarrow\infty$. Note that $g$ is bounded and $a_n\rightarrow\infty$. Then (\ref{eq normal})  follows from the Lindeberg-Feller central limit theorem (e.g., \cite[Theorem 6.13]{kallenberg2021foundations}).

Next, write the left-hand side and the right-hand side of \eqref{eq:joint} respectively as 
$Y_n(h)$ and $Y(h)$. Note that both are linear in $h$.  
Now suppose that for each $k=1,\ldots,K$, a sequence $h_j^{(k)}\in \cl{S}^{k}$, $j  \in \bb{N}_+$ satisfies
\begin{equation}\label{eq:h_j approx}
    \left\|h_j^{(k)}-h^{(k)}\right\|_{L^2(\mu^k)}\rightarrow 0
\end{equation} as $j\rightarrow\infty$.    Write $h_j=(h_j^{(0)},h_j^{(1)},\ldots,h_j^{(K)},0,\ldots)\in H$.  
Observe that each $I_k^{(n)}(h_j^{(k)})$ is a polynomial of $W_n(B_i)$'s for some $B_i\in \cl{E}_0$, $i=1,\ldots,m$, $m\in \bb{N}_+$, and each $I_k(h_j^{(k)})$ is the same polynomial of $W(B_i)$'s, $k=1,\ldots,K$. Recall also that the zeroth order integrals are understood as the same constant. 
Then by \eqref{eq: conv for W} and the continuous mapping theorem, we have,
for each fixed $j$, as $n\rightarrow\infty$.
$
Y_n(h_j)\Rightarrow Y(h_j).
$
On the other hand, by  the relations \eqref{eq38} and \eqref{eq:mwi 2nd moment}, we have as $j\rightarrow\infty$ that 
\[
\|Y(h_j) - Y(h)\|_{L^2(\bb{P})} ^2=\|Y(h_j- h)\|_{L^2(\bb{P})}^2 = \sum_{k=0}^K   k!  \left\|   \wt{h}_j^{(k)}-\wt{h}^{(k)} \right\|_{L^2(\mu^k)}^2 \rightarrow 0.
\] 
Therefore, in view of a standard triangular approximation argument (e.g., \cite[Theorem 3.2]{billingsley1968convergence}), it suffices to show  
\begin{equation}\label{eq:tri approx}
\lim_{j\rightarrow\infty}\limsup_{n\rightarrow\infty}   \| Y_n(h_j)- Y_n(h) \|_{L^2(\bb{P})} = 0.
\end{equation}
Indeed, by a triangular inequality, Proposition \ref{prop2} and  \eqref{eq:mwi 2nd moment}, we have
\begin{align*}
\limsup_n\| Y_n(h_j)- Y_n(h) \|_{L^2(\bb{P})}\le \sum_{k=0}^K   \limsup_n \left\| I_k^{(n)}(h_j^{(k)}-h^{(k)}) \right\|_{L^2(\bb{P})} =  \sum_{k=0}^K   \sqrt{k!}  \left\|   \wt{h}_j^{(k)}-\wt{h}^{(k)} \right\|_{L^2(\mu^k)}.
\end{align*}
Then \eqref{eq:tri approx} follows from \eqref{eq:h_j approx}  and \eqref{eq:mwi 2nd moment} by letting $j\rightarrow\infty$ in the last bound above.

\end{proof}

\begin{proof}[Proof of Theorem \ref{thm7} ]
    Let \[
    R_{m,n} = \left\| \sum_{m < k \leq K_n} I_k^{(n)}(h^{(k)}) \right\|^2_{L^2(\bb{P})},\quad m,n\in \bb{N}_+.
    \]  First observe that $\sum_{k=0}^m I_k(h^{(k)})\rightarrow \sum_{k=0}^\infty I_k(h^{(k)})$  in $L^2(\bb{P})$ as $m\rightarrow\infty$ in view of \eqref{eq38}, \eqref{eq:mwi 2nd moment} and \eqref{eq:H norm}.  Combining this with Lemma \ref{thm6} and a triangular approximation argument, it suffices to show that there exists $  K_n \nearrow \infty$ as $n\rightarrow\infty$, so that
    \begin{equation}\label{eq:Rmn tri approx}
        \lim_{m\rightarrow \infty} \limsup_{n\rightarrow\infty} R_{m,n} =0.
    \end{equation}

Let $\gamma_n(k_1,k_2)  = \mathbb{E}[I^{(n)}(h^{(k_1)})I^{(n)}(h^{(k_2)})]$. 
Suppose two sequences $\{b_K\}_{K\in \bb{N}_+}$  and $\{s_K\}_{K\in \bb{N}_+}$ satisfy $b_K = o(K^{-2})$ and $s_K = o(K^{-1})$ as $K\rightarrow\infty$.
By Proposition \ref{prop2},  for any $ K \in \mathbb{Z}_{+}$, there exists  $M(K)\in \bb{N}_+$,  such that  whenever $ n \geq M(K)$, 
\begin{align}
   & \sup_{1 \leq k_1<k_2 \leq K} \left|\gamma_n(k_1,k_2)\right| \leq b_K,
   \label{eq:supdiff}\\
   & \sup_{1 \leq k \leq N} \left|\gamma_n(k,k) - k! \| \wt{h}^{(k)}\|^2_{L^2(\mu^k)} \right| \leq s_K.
   \label{eq:supsame}
\end{align}
Assume without loss of generality, $M(0) = 1 < M(1) < M(2)< M(3)\ldots$. Consider $M(K)$ as a function of $K$ and define $K_n : = \max \{m \geq 0: n \geq M(m) \}$.   Then $K_n \nearrow \infty$ as $n \rightarrow \infty$.
 We have by (\ref{eq:supdiff}) and (\ref{eq:supsame}) that
\begin{align*}
    R_{m,n} & = \sum_{m <k_1,k_2 \leq K_n} \gamma_n(k_1,k_2)  \leq  \sum_{m < k \leq K_n} \gamma_n(k,k) + 2\sum_{m \leq k_1<k_2 \leq K_n} \left|\gamma_n(k_1,k_2)\right|\\
    & \leq \sum_{m \leq k \leq K_n} \left[ k! \|\wt{h}^{(k)}\|_{L^2(\mu^k)}^2 + s_{K_n} \right] + 2\sum_{m \leq k_1<k_2 \leq K_n} b_{K_n}\\
    & \leq \sum_{m \leq k \leq K_n} k! \|\wt{h}^{(k)}\|_{L^2(\mu^k)}^2 + s_{K_n} K_n + b_{K_n}K_n^2 = : S_{m,n}.
\end{align*}
Note that by the restrictions on $b_K$ and $s_K$ and \eqref{eq:H norm}, we have
$
    \lim_{m} \limsup_{n } S_{m,n} = 0
$.
Hence the desirable relation \eqref{eq:Rmn tri approx} follows. 
\end{proof}

\begin{proof}[Proof of Proposition \ref{prop4}]
We shall develop a more precise estimate of the left-hand side of (\ref{eq:supdiff}) and (\ref{eq:supsame}). 
We will use the  following   Stirling-type estimate  $ C_1 m^{m+1/2} / e^{m}\le m!\le C_2 m^{m+1/2} / e^{m}$, $m\in \bb{N}_+$, $0<C_1<C_2$, in several places below. First, in view of \eqref{eq:91}, we have for $k_1,k_2\in \bb{N}_+$,  $k_1<k_2$, that 
 \begin{align}
  \left|\gamma_n\left(k_1, k_2\right)\right|   
\leq     \sum_{l=0}^{k_1}  \sum_{p=0}^l  &  n^{-\frac{l-p}{2}}  \left| \mathcal{B}\left(l, p, k_1, k_2\right)\right| \ \times \frac{C^{k_1+k_2-l-p}\left(k_1+k_2-l-p\right) !}{\left(k_1+k_2-l-p\right)^{\left(k_1+k_2-l-p\right) / 2}}   \left|F_l^{(n)}(h^{(k_1)}, h^{(k_2)})\right|.\label{eq42}
 \end{align}
 Observe   
 for $0 \le l \le p\le \min(k_1,k_2)$, $1\le k_1,k_2\le K$,  that
\begin{equation}\label{eq:B(l,p) bound}
\left| \cl{B} \left(l, p, k_1, k_2\right) \right| =   \binom{k_1}{l}\binom{k_2}{l}\binom{l}{p} l!   \leq 2^{k_1+k_2+l} l!  \le C_0^{K} K^{K},
\end{equation}
and  
\begin{equation}\label{eq:coef bound}
\frac{C^{k_1+k_2-l-p}\left(k_1+k_2-l-p\right) !}{\left(k_1+k_2-l-p\right)^{\left(k_1+k_2-l-p\right) / 2}}  \le C_0^K K^{K}.
\end{equation}
 
On the other hand, recalling $\mu_n=(n/a_n)P_n\le \mu$ under Assumption \ref{asu2}, we have for $0\le l\le k_1<k_2\le K$ that 
\begin{align*}
     &\big|  F^{(n)}_l({h}^{(k_1)}, {h}^{(k_2)}) \big|=\big|  F^{(n)}_l(\wt{h}^{(k_1)},\wt{h}^{(k_2)}) \big| \\
     & \leq    \left( \frac{a_n}{n}\right)^{(k_1+k_2)/2 -l} 
 \frac{1}{|\mathcal{B}(l,k_1,k_2)|} \, \mu_n^{ k_1+k_2-l} \left( \sum_{B(\mathcal{N}) \in \mathcal{B}(l,k_1,k_2)} \left|  {\left(\wt{h}^{(k_1)} \otimes \wt{h}^{(k_2)} \right) }_{B(\mathcal{N})} \right| \right)  \\
  \leq   & \left(\frac{a_n}{n}\right)^{\left(k_1+k_2\right) / 2-l}  \frac{1}{\left|\mathcal{B}\left(l, k_1, k_2\right)\right|} \sum_{B(\mathcal{N}) \in \mathcal{B}(l,k_1,k_2)} \left\| {\left(\wt{h}^{(k_1)} \otimes \wt{h}^{(k_2)} \right) }_{B(\mathcal{N})}\right\|_{L^1(\mu^{k_1+k_2-l})}.
\end{align*}
Note that  by the assumption that $\operatorname{supp}(h^{(k)}) \subset E_{n_0}^k$ for all $k\in \bb{N}_+$, the Cauchy–Schwarz inequality and \eqref{contraction2} (with $P_n$ replaced by $\mu$), we have 
\begin{align}
 \left\| {\left(\wt{h}^{(k_1)} \otimes \wt{h}^{(k_2)} \right) }_{B(\mathcal{N})}\right\|_{L^1(\mu^{k_1+k_2-l})} 
\le &     \left\| {\left(\wt{h}^{(k_1)} \otimes \wt{h}^{(k_2)} \right) }_{B(\mathcal{N})}\right\|_{L^2(\mu^{k_1+k_2-l})}   \mu(E_{n_0})^{(k_1+k_2-l)/2}\notag \\ \le & \left\|  \wt{h}^{(k_1)} \right\|_{L^2(\mu^{k_1})}   \left\| \wt{h}^{(k_2)}\right\|_{L^2(\mu^{k_1})}   \mu(E_{n_0})^{(k_1+k_2-l)/2}
\leq   \frac{C_1 \mu(E_{n_0})^{(k_1+k_2-l)/2}}{\sqrt{k_1!k_2!}},\notag
\end{align}
where in the last inequality we used the assumption on $H$ in Definition \ref{defnofH}.  Therefore for  $0\le l\le k_1<k_2\le K$, under which $(k_1+k_2)/2-l\ge 1/2$, we have
\begin{align}
  &\left|  F^{(n)}_l(h_{k_1},h_{k_2}) \right|\le  C_1 \left( \frac{a_n}{n}\right)^{(k_1+k_2)/2 -l}  \frac{ \mu(E_{n_0})^{(k_1+k_2-l)/2}}{\sqrt{k_1!k_2!}}\le C_2 \pp{\frac{a_n}{n}}^{1/2}.\label{eq:h contra bound}
\end{align}
Combining \eqref{eq:B(l,p) bound}, \eqref{eq:coef bound} and \eqref{eq:h contra bound} into \eqref{eq42}, we derive
\begin{align}
   \sup _{1 \leq k_1<k_2 \leq K}\left|\gamma_n\left(k_1, k_2\right)\right|  
  \leq  C^K  K^{2K} \left(\frac{a_n}{n}\right)^{1/2}. \label{eq32}
\end{align}

Next in view of \eqref{eq:T0}, we can write for $k\in \bb{N}_+$ that
\begin{align}
  &\gamma_n(k, k)-k !\|\wt{h}^{(k)}\|_{L^2\left( \mu^k\right)}^2 \notag\\  = &    \sum_{ \substack{ 0\leq l < k \\0\leq  p \leq l}}    n^{-\frac{l-p}{2}} |\mathcal{B}(l, p, k, k)|(2 k-l-p) ! B_{n, 2 k-l-p}  F_l^{(n)}\left(h^{(k)}, h^{(k)}\right) \notag\\
& +  \sum_{p=0}^{k-1}   n^{-(k-p) / 2} |\mathcal{B}(k, p, k, k)|  (k-p)! B_{n, k-p} F_k^{(n)}\left(h^{(k)}, h^{(k)}\right)\notag   + \left( F_k^{(n)}\left(h^{(k)}, h^{(k)}\right) - k !\|\wt{h}^{(k)}\|_{L^2\left( \mu^k\right)}^2 \right)  \\
&=:  \mathrm{I}_n(k)+ \mathrm{II}_n(k)+ \mathrm{III}_n(k).\label{eq:I+II+III}
\end{align}
 Similarly as \eqref{eq:h contra bound} we have
$
\sup_{1 \leq k \leq K} \sup _{0 \leq l < k} \left|F_l^{(n)}(h^{(k)}, h^{(k)})\right|\le C \left(\frac{a_n}{n}\right)^{1 / 2}.
$
This together with \eqref{eq:B(l,p) bound},   \eqref{eq:B_n,k bound} (assume  $n\ge K/2$ without loss of generality) and \eqref{eq:coef bound}    implies 
\begin{equation}\label{eq:I}
\sup_{1\le k\le K}|\mathrm{I}_n(k)|\le  C^K K^{2K}  \left(\frac{a_n}{n}\right)^{1 / 2}.
\end{equation}
For   $\mathrm{II}_n(k)$ in \eqref{eq:I+II+III},  first note that by \eqref{eq:B_n,k bound}, for $0\le p<k\le K$, we have
$
(k-p)!B_{n,k-p}\le C^K K^{K/2}.     
$
Combining this with  $n^{-(k-p)/2}\le n^{-1/2}$,    applying  Lemma \ref{Lem:Fnl bounded} for the term $F_k^{(n)}\left(h^{(k)}, h^{(k)}\right)$, as well as  the relation  \eqref{eq:B(l,p) bound} for the term $|\mathcal{B}(k,p,k,k)|$,
we have
\begin{equation}\label{eq:II}
\sup_{1\le k\le K}|\mathrm{II}_n(k)|\le  C^K K^{(3/2)K}  \left(\frac{ 1}{n}\right)^{1 / 2}.
\end{equation}
Since $\operatorname{supp}(h^{(k)}) \subset E_{n}^k$ for all $k\in \bb{N}_+$, $n\ge n_0$,   based on the setup of Example \ref{exam1} we see that
\begin{equation}\label{eq:III}
\mathrm{III}_n(k)=0,\quad n\ge n_0, \ k\ge 1.
\end{equation}
Combining \eqref{eq:I} \eqref{eq:II} and \eqref{eq:III} to \eqref{eq:I+II+III}, we derive that for $n\ge n_0$, 
\begin{equation}\label{eq:supsame est}
\sup_{1\le k\le K} \left|\gamma_n(k, k)-k !\| \wt{h}^{(k)}\|_{L^2\left( \mu^k\right)}^2\right| \le  C^K K^{2K}  \left(\frac{a_n}{n}\right)^{1 / 2} +  C^K K^{(3/2)K}  \left(\frac{ 1}{n}\right)^{1 / 2}.
\end{equation}

Finally,  to satisfy the restrictions $b_K = o(K^{-2})$ and $s_K = o(K^{-1})$ in 
   \eqref{eq:supdiff} and \eqref{eq:supsame} respectively,  in view of  \eqref{eq32}  and \eqref{eq:supsame est}, it suffices for $K=K_n$ to satisfy 
\[
K \ln C  + 2K \ln( K )  + 2\ln(K) - \frac{1}{2} \ln\left(\frac{n}{a_n}\right) \rightarrow -\infty
 \]
 and
 \[  
 K\ln C + \frac{3}{2}K \ln (K) +\ln(K)  - \frac{1}{2}\ln(n)   \rightarrow-\infty 
\]
as $K\rightarrow\infty$.  With the bound $\ln(K)\le CK^{\epsilon_0}$ for any $\epsilon_0>0$, $K\ge 1$, it can be  verified that the rate claimed in the proposition satisfies the relations above.
\end{proof}

\end{document}